\documentclass{paper}

\usepackage{hyperref}
\usepackage{bm}
\usepackage{amsmath}
\usepackage{amsthm}
\usepackage{amsfonts}
\usepackage{graphicx}
\usepackage[usenames,dvipsnames]{xcolor}

\newtheorem{theorem}{Theorem}[section]
\newtheorem{remark}{Remark}[section]

\newcommand{\mcU}{\mathcal{U}}

\newcommand{\mcD}{\mathcal{D}}
\newcommand{\mcG}{\mathcal{G}}

\newcommand{\mcX}{\mathcal{X}}
\newcommand{\mcN}{\mathcal{N}}
\newcommand{\mcL}{\mathcal{L}}

\newcommand{\mbR}{\mathbb{R}}
\newcommand{\mbRd}{{\mathbb{R}^d}}
\newcommand{\mbRn}{{\mathbb{R}^n}}

\newcommand{\omg}{{\Omega}}
\newcommand{\omgi}{{\Omega_I}}
\newcommand{\omgn}{{\Omega_N}}
\newcommand{\omgd}{{\Omega_D}}
\newcommand{\omgomgi}{{\Omega\cup\Omega_I}}
\newcommand{\oomg}{{\overline\Omega}}

\newcommand{\ds}{\displaystyle}

\def \alphab{{\boldsymbol\alpha}}

\def \nub{{\boldsymbol \nu}}

\def \xb{\bm{x}}
\def \yb{\bm{y}}

\def \veps{\varepsilon}
\def \sl{{s_l}}
\def \ul{{u_l}}
\def \un{{u_n}}
\def \gn{{g_n}}
\def \wun{{\widetilde u_n}}

\def \wgn{{\widetilde g_n}}

\title{A physically-consistent, flexible and efficient strategy to convert local boundary conditions into nonlocal volume constraints}

\begin{document}
\maketitle

\vspace{-1cm}
\noindent \textsf{M. D'Elia}
\textsf{\textit{Sandia National Laboratories, Albuquerque, NM}}\\
\textsf{X. Tian,} \textsf{\textit{University of Texas, Austin, TX}}\\
\textsf{Y. Yu,} \textsf{\textit{Lehigh University, Bethlehem, PA}}

\vspace{.5cm}
\noindent
\textsf{\textbf{Abstract.}}

\noindent Nonlocal models provide exceptional simulation fidelity for a broad spectrum of scientific and engineering applications. However, wider deployment of nonlocal models is hindered by several modeling and numerical challenges. Among those, we focus on the nontrivial prescription of nonlocal boundary conditions, or volume constraints, that must be provided on a layer surrounding the domain where the nonlocal equations are posed. The challenge arises from the fact that, in general, data are provided on surfaces (as opposed to volumes) in the form of force or pressure data. In this paper we introduce an efficient, flexible and physically consistent technique for an automatic conversion of surface (local) data into volumetric data that does not have any constraints on the geometry of the domain and on the regularity of the nonlocal solution and that is not tied to any discretization.
We show that our formulation is well-posed and that the limit of the nonlocal solution, as the nonlocality vanishes, is the local solution corresponding to the available surface data. Quadratic convergence rates are proved for the strong energy and $L^2$ convergence. We illustrate the theory with one dimensional numerical tests whose results provide the ground work for realistic simulations. 

\vspace{.5cm}

\vspace{.5cm}

\section{Introduction and motivation}\label{sec:introduction}
Nonlocal models employ integral rather than differential operators which allows them to relax the regularity constraints of partial differential equations (PDEs) and to capture effects arising from long-range forces at the microscale and mesoscale, not accounted for by PDEs. Consequently, nonlocal models provide exceptional simulation fidelity for a broad spectrum of applications such as 
fracture mechanics \cite{Ha2011,Littlewood2010,Silling2000}, 
anomalous subsurface transport \cite{Benson2000,Schumer2003,Schumer2001}, 
phase transitions \cite{Bates1999,Delgoshaie2015,Fife2003}, 
image processing \cite{Buades2010,Gilboa2007,Gilboa2008,Lou2010}, 
multiscale and multiphysics systems \cite{Alali2012,Askari2008}, 
MHD \cite{Schekochihin2008}, 
and stochastic processes \cite{Burch2014,DElia2017,Meerschaert2012,MeKl00}.

The main difference between PDE models and the nonlocal models we consider is that, in the former case, interactions between two domains only occur due to contact, whereas in the latter case, interactions can occur at a distance. In this work, for simplicity of the exposition and without loss of generality (see Remark \ref{rem:extension}), we consider the nonlocal counterpart of elliptic differential operators. In its simplest form, the action of a nonlocal diffusion operator on a scalar function $u:\mbRn\to\mbR$ is given by
\begin{displaymath}
\mcL u(\xb)=C\int_\mbRd \big(u(\yb)-u(\xb)\big) \,\gamma (\xb,\yb )\,d\yb \qquad  \xb\in\mbRd,
\end{displaymath}
where the {\it kernel} function $\gamma$, usually with bounded support, is related to the specific application and determines the smoothing properties of $\mcL$. The integral form above allows us to catch long-range interactions so that every point in a domain interacts with a neighborhood of points. Also, such form reduces the regularity requirements for the solution, which is able to describe discontinuous (for e.g. fracture mechanics) or anomalous (for e.g. subsurface dispersion) behaviors.

However, the increased accuracy of nonlocal models comes at a price: several modeling and numerical challenges arise. These include the nontrivial prescription of ``nonlocal'' boundary conditions, the often prohibitively expensive numerical solution and the definition of model parameters (such as $\gamma$), often unknown or subject to uncertainty. All these (open) problems can hinder wider deployment of nonlocal models and are the subject of current research in the fast-growing nonlocal community. In this work we focus on the first challenge. 

Because of nonlocal interactions, when solving a nonlocal problem in a bounded domain, the prescription of classical boundary conditions does not guarantee the well-posedness of the equations \cite{Du2012}; in fact, in general, {\it nonlocal} boundary conditions, or, more properly, volume constraints, must be defined on a layer surrounding the domain. However, it is often the case that such information is not available, whereas it is easy to measure surface (local) data. Consequently, one of the biggest challenges to be addressed before nonlocal models can be widely applied in realistic contexts is the {\it conversion of local boundary conditions, defined on surfaces, into volume constraints, defined on volumes}. 

Previous attempts to tackle the conversion are either too expensive (solving an optimization problem) or too restrictive (requiring conditions on geometry or dimensionality).

The first approach is an optimization-based coupling method that mimics generalized overlapping domain-decomposition formulations \cite{DElia2018-P}. The main idea is to decompose the domain into a local and nonlocal subdomains where the former is placed in a neighborhood of the part of the boundary where only surface data are available. This choice allows both the local and nonlocal problems to be well-posed and circumvents the prescription of volume constraints when not available. On the other hand, this method requires the solution of a nonlocal minimization problem whose algorithm may require several computation of the nonlocal solutions, dramatically increasing the computational effort.

Paper \cite{Cortazar2008} is the first that interprets the nonlocal Neumann boundary condition as a body force acting on the boundary layer of the domain, where $L^1$ convergence of nonlocal solutions to the corresponding local ones is shown. Later in \cite{Tao2017}, a careful modification of the body force in a one dimensional setting is found that leads to a second order uniform convergence of solutions as the nonlocal interaction vanishes. The second order convergence result is then extended to two dimensions in \cite{YuNeumann2019}, where the curvature of the computation domain plays an important role in the definition of the modified body force. 
Recently, \cite{DuZhangZheng} achieves the second order uniform convergence in one dimension with another approach. 
To the best of our knowledge, no work has yet discussed second order nonlocal approximations to the local Neumann boundary value problems in space dimension higher than two. The complexity of geometric bodies to be dealt with in high dimension is an obvious hindrance. 

We propose a computationally cheap, flexible and physically consistent method for an efficient conversion that has no constraints on dimensionality, geometry, regularity of the nonlocal solution and that is not tied to any discretization. Our main and most promising approach consists of three simple steps.
\begin{itemize} 
\item[\bf A] Solution of a computationally cheap local model using available surface data.
\item[\bf B] Derivation, from {\bf A}, of forces corresponding to the local solution in the thick nonlocal layer.
\item[\bf C] Solution of the nonlocal model using the forces derived in {\bf B}.
\end{itemize}

\noindent Note that the forces computed in {\bf B} are equivalent to {\it nonlocal} Neumann data, which is used in {\bf C} as volume constraint for the solution of the nonlocal problem. Also note that local and nonlocal problems are completely uncoupled; this feature becomes very powerful when dealing with large scale problems (as it is often the case in engineering applications); in fact, local and nonlocal solvers can be used as black boxes and the overall cost of the proposed method is the same of a nonlocal problem, for given nonlocal boundary data. This is due to the fact that the cost of solving the local problem is negligible compared to the one of the nonlocal problem. Note that the uncoupling of local and nonlocal equations allows for completely independent discretizations of the local and nonlocal equations\footnote{As an example, one can use a mesh-free discretization for the nonlocal models and a mesh-based one for the local model.}. In fact, application of the nonlocal operator to the discretized local solution in step {\bf B} only requires projection of the latter onto the nonlocal discretization space. 
Furthermore, this approach is such that the nonlocal solution computed in {\bf C} reduces to the solution computed in {\bf A}, as the nonlocal interactions vanish, with a quadratic rate of convergence for both the (nonlocal) energy and $L^2$ norms with respect to the characteristic interaction length. 

A few considerations are in order. Even though we do not require additional regularity of the nonlocal solution, we do assume that the given surface data is such that the corresponding local problem computed in {\bf A} is well-posed (for, e.g., the classical Poisson equation square integrability over the boundary of the force/pressure data is enough to guarantee the existence and uniqueness of the local solution). We also mention that in the analysis of the asymptotic behavior of the nonlocal solution for vanishing nonlocality we assume that the local solution belongs to $C^4$. However, this additional regularity is not required in practice.

We expect the proposed strategy to advance the state of the art for predictive nonlocal modeling by providing an efficient in-demand tool that will impact a broad class of applications and unlock the full potential of nonlocal models.

Note that we also introduce an alternative, more straightforward, strategy that has exactly the same properties of the approach described in {\bf A}--{\bf C}, but delivers solutions whose behavior is closer to the local one.

\smallskip The paper is organized as follows. In the following section we introduce the notation and recall relevant results of the nonlocal vector calculus, a theory developed in the last decade by Du et al. \cite{Du2013} that allows one to study nonlocal diffusion problems in a very similar way as PDEs by framing nonlocal equations in a variational setting. In Section \ref{sec:strategy} we introduce two alternative strategies to the conversion problem, discuss their properties, and provide a qualitative comparison. In Section \ref{sec:local-limit} we study the convergence to the local limit of the nonlocal solution for the most promising strategy and show quadratic strong convergence in both the nonlocal energy norm and $L^2$ norm. In Section \ref{sec:numerical-tests} we illustrate the theoretical results in a one-dimensional setting.

\section{Preliminaries}\label{sec:preliminaries}
In this section we introduce the nonlocal vector calculus and recall results relevant to this paper. Let $\omg$ be a bounded open domain in $\mbRd$, $d=1,2,3$, with Lipschitz-continuous boundary $\partial\omg$ and $\alphab(\xb,\yb) \colon \mbRd\times\mbRd\to \mbRd$ be an antisymmetric function, i.e. $\alphab(\yb,\xb)=-\alphab(\xb,\yb)$. For the functions $u(\xb)\colon \mbRd \to \mbR$ and $\nub(\xb,\yb)\colon \mbRd\times\mbRd\to \mbRd$ we define the nonlocal divergence $\mcD\colon \mbRd \to \mbR$ of $\nub(\xb,\yb)$ as

\begin{equation}\label{ndiv}
\mcD\big(\nub\big)(\xb) := \int_{\mbRd} \big(\nub(\xb,\yb)+\nub(\yb,\xb)\big)\cdot\alphab(\xb,\yb)\,d\yb\qquad \xb\in\mbRd
\end{equation} 
and the nonlocal gradient $\mcG\colon  \mbRd\times\mbRd\to\mbRd$ of $u(\xb)$ as 
\begin{equation}\label{ngra}
\mcG\big(u\big)(\xb,\yb) := \big(u(\yb)-u(\xb)\big)  \alphab(\xb,\yb) \qquad \xb,\yb\in\mbRd.
\end{equation}
It is shown in \cite{Du2013} that the adjoint $\mcD^*=-\mcG$. Next, we define the nonlocal diffusion $\mcL\colon \mbRd \to \mbR$ of $u(\xb)$ as a composition of the nonlocal divergence and gradient operators, i.e.
\begin{equation}\label{eq:L}
\mcL u(\xb) :=  \mcD\big(\mcG u\big)(\xb) = 
                2\int_\mbRd \big(u(\yb)-u(\xb)\big) \,\gamma (\xb,\yb )\,d\yb \qquad  \xb\in\mbRd,
\end{equation}
where  $\gamma(\xb,\yb):=\alphab(\xb,\yb)\cdot \alphab(\xb,\yb)$ is a non-negative symmetric kernel\footnote{There are more general representations of the nonlocal diffusion operator, these are associated with nonsymmetric and not necessarily positive kernel functions. In such cases $\mcL$ may define a model for non-symmetric diffusion phenomena, we mention e.g. nonsymmetric jump processes \cite{DElia2017}.}. Note that this is the same operator introduced in Section \ref{sec:introduction}.
We define the interaction domain of an open bounded region $\omg\in\mbRd$ as
\begin{displaymath}
\omgi = \{\yb\in\mbRd\setminus\omg: \; \gamma(\xb,\yb)\neq 0, \; \xb\in \omg\},
\end{displaymath}
and set $\oomg =\omg\cup\omgi$. This domain contains all points outside of $\omg$ that interact with points inside of $\omg$; as such, $\omgi$ is the volume where nonlocal boundary conditions must be prescribed to guarantee the well-posedness of nonlocal euqations (see Section \ref{sec:volume-constrained-problems}).
We make the following assumptions: for $\xb\in\omg$
\begin{displaymath}
\left\{
\begin{array}{ll}
\gamma(\xb,\yb)  >  0 \quad &\forall\, \yb\in B_\varepsilon(\xb)\\[2mm]
\gamma(\xb,\yb)  = 0 \quad &\forall\, \yb\in {\oomg} \setminus B_\varepsilon(\xb),
\end{array}\right.
\end{displaymath}
where $B_\veps(\xb) = \{\yb\in{\oomg}: \; \|\xb-\yb\|<\veps,\; \xb\in \omg\}$ and $\veps$ is the interaction radius or horizon. For such kernels the interaction domain is a layer of thickness $\veps$ that surrounds $\omg$, i.e.
\begin{equation}\label{omgie}
\omgi = \{ \yb\in \mbRd\setminus\omg: \; \|\yb-\xb\|<\varepsilon, \;\xb\in\omg\}.
\end{equation}
We refer to Figure \ref{fig:domains} (left) for an illustration of a two-dimensional domain, the support of $\gamma$ and the induced interaction domain.
\begin{figure}
\centering
\begin{tabular}{cc}
\includegraphics[width=0.35\textwidth]{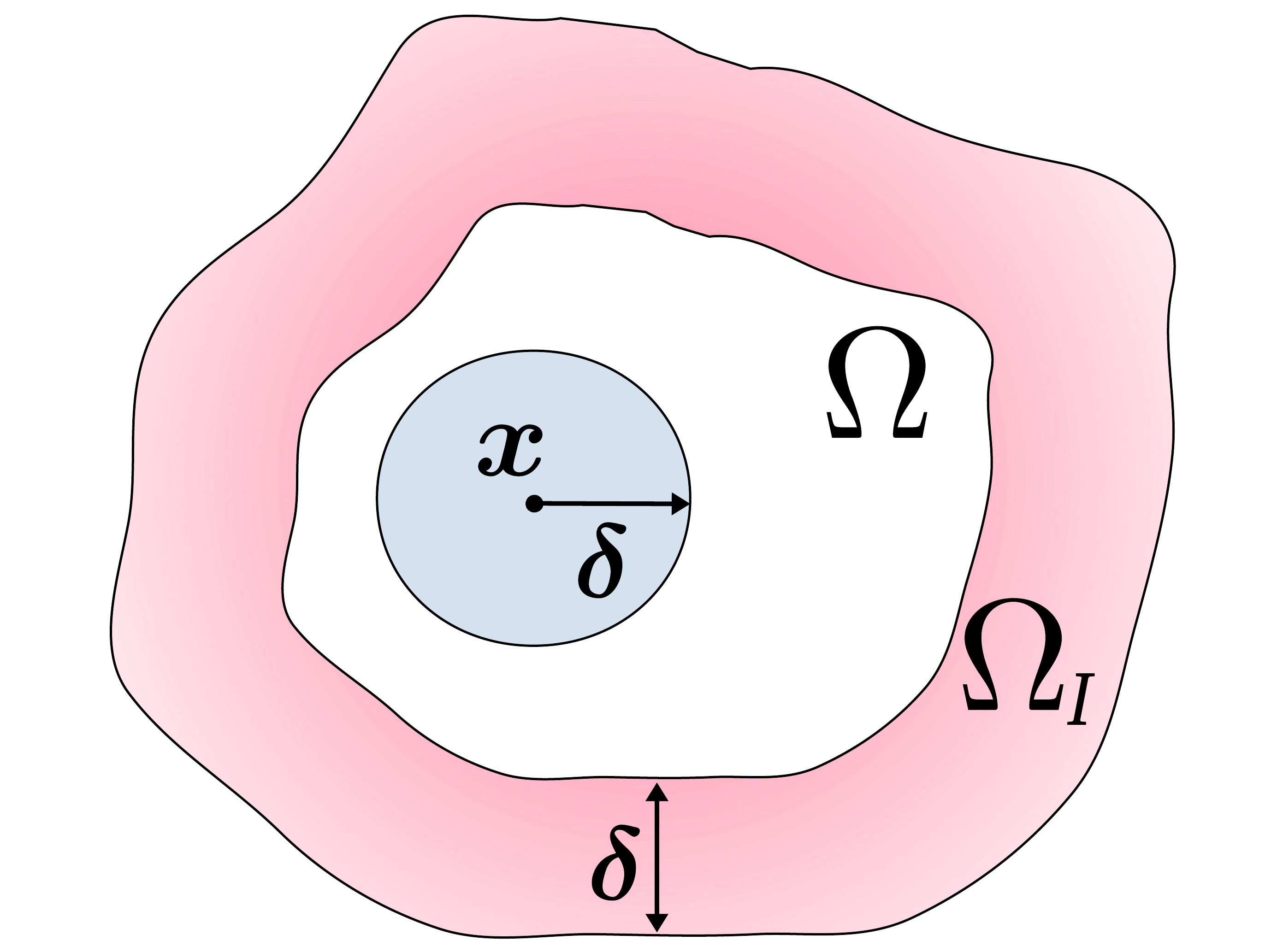}
\hspace{1cm}
\includegraphics[width=0.35\textwidth]{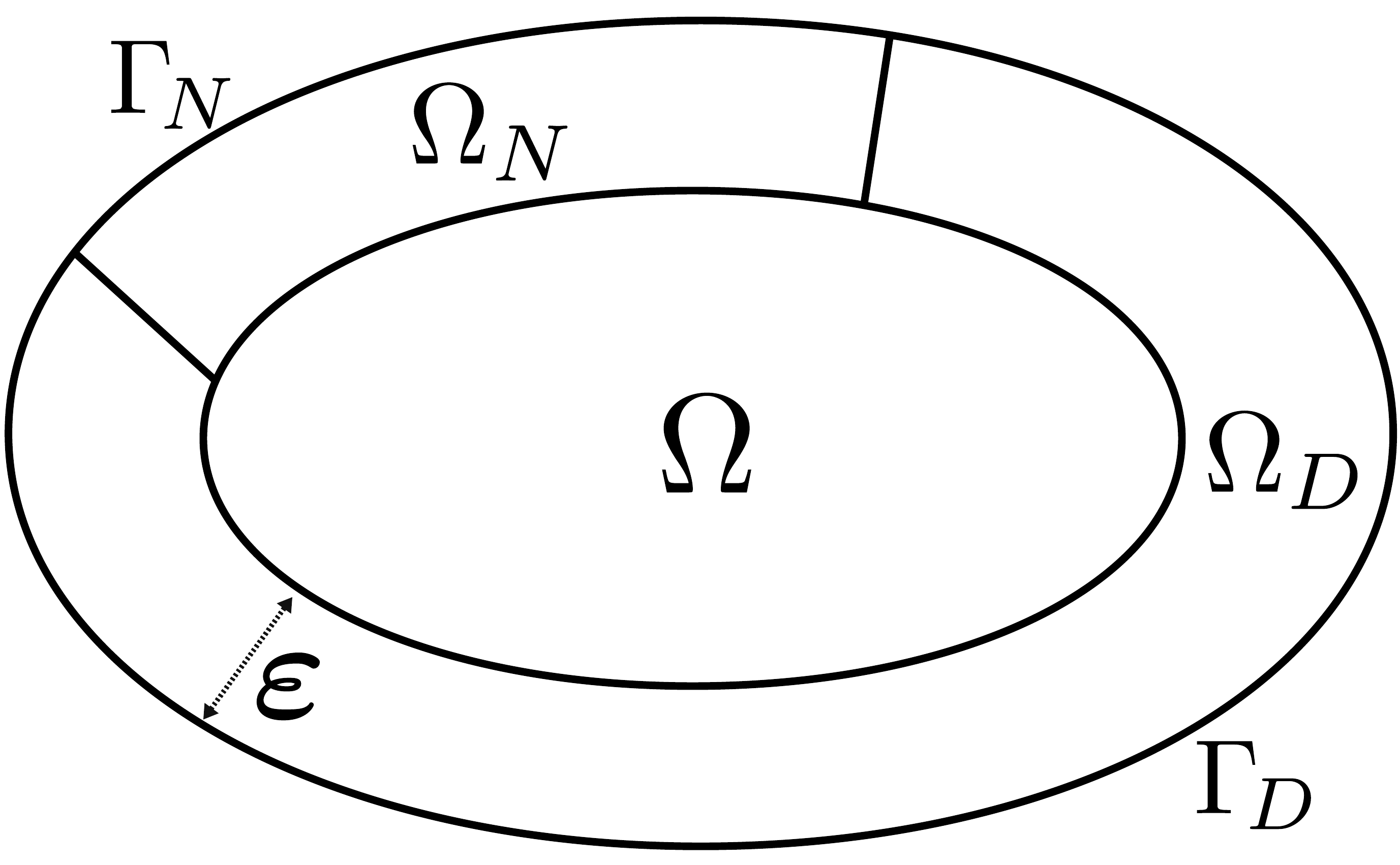}
\end{tabular}
\vspace{-0ex}
\caption{Left: the domain $\omg$, the support of $\gamma$ at a point $\xb\in\omg$, $B_\delta(\xb)$, and the induced interaction domain $\omgi$. Right: two-dimensional configuration. Here, $\omgn\cup\omgd=\omgi$, $\omgomgi=\oomg$ and $\Gamma_N\cup\Gamma_D=\Gamma$.}
\label{fig:domains}
\end{figure}
Corresponding to the divergence operator $\mcD(\nub)$ we introduce a nonlocal {\it interaction} operator 
\begin{equation}\label{eq:interaction}
\mcN(\nub)(\xb)=-\int_\oomg \left(\nub(\xb,\yb)+\nub(\yb,\xb)\right)\alphab(\xb,\yb)\,d\yb \qquad \xb\in\omgi.
\end{equation}
The integral $\int_\omgi \mcN(\nub)\,d\xb$ generalizes the notion of a flux $\int_{\partial\omg}{\bf q}\cdot{\bf n}\,dA$ through the boundary of a domain, with $\mcN(\nub)$ playing the role of a flux density ${\bf q}\cdot{\bf n}$. The key difference between \eqref{eq:interaction} and a conventional flux is that in the former the flux is a {\it volume} integral, whereas in the latter it is a {\it boundary} integral. Nonetheless, the nonlocal divergence and interaction operators satisfy a nonlocal Gauss theorem $\int_\omg\mcD(\nub)\,d\xb=\int_\omgi\mcN(\nub)\,d\xb$. We refer to \cite{Du2013} for additional nonlocal vector calculus results, including  generalized nonlocal Green's identities.

We respectively introduce the nonlocal energy semi-norm, nonlocal energy space, and nonlocal volume-constrained energy space
\begin{equation}
\begin{array}{ll}
& |||v|||^2    := \displaystyle\frac12\int_{\oomg}\int_{{\oomg}}(\mcG v)^2\,d\yb \, d\xb \\ [5mm]
& V(\oomg)   := \left\{ v  \in L^2(\oomg) \,\,:\,\, |||v|||_{\oomg} < \infty \right\}\\[3mm]
& V_c(\oomg) := \left\{v\in V({\oomg}) \,\,:\,\, v=0\;{\rm on}\;\omgd\right\} \;\; 
                  \hbox{ for $\omgd \subseteq \omgi$.}
\end{array}
\end{equation}
We also define the volume-trace space $\widetilde V_c(\oomg):=\{v|_\omgd: \,v\in V(\oomg)\}$ and the dual spaces $V'(\oomg)$ and $V'_c(\oomg)$ with respect to $L^2$-duality pairings.

We consider kernels such that the corresponding energy norm satisfies a Poincar\'e-like inequality, i.e. $\|v\|_{0,{\oomg}}\leq C_{pn}|||v|||$ for all $v\in V_c(\oomg)$, where  $C_{pn}$ is referred to as the nonlocal Poincar\'e constant. Kernels satisfying this property can be found in \cite[Section~4.2]{Du2012}; for such kernels\footnote{The nonlocal Poincar\'e inequality holds for an even more general class of properly scaled, non-increasing, kernel functions, see \cite{DuMengesha}.}, in \cite{DuMengesha}, it is shown that the Poincar\'e constant is independent of $\veps$ if $\veps\in (0, \veps_0]$ with a certain fixed number $\veps_0$.

\smallskip
A popular example is the class of integrable kernels\footnote{Specifically, we are referring to kernels for which there exist positive constants $\gamma_1$ and $\gamma_2$ such that $\gamma_1\leq \int_{\oomg\cap B_\veps(\xb)} \gamma(\xb,\yb)\,d\yb$ and $\int_\oomg \gamma^2(\xb,\yb)\,d\yb\leq \gamma_2^2$ for all $\xb\in\omg$.} for which $V(\oomg)$ and $V_c(\oomg)$ are equivalent to $L^2({\oomg})$ and $L^2_c(\oomg)$; in this case, the operator $\mcL$ is such that $\mcL:L^2(\oomg)\to L^2(\oomg)$ \cite{Du2012}.

\subsection{Volume-constrained nonlocal diffusion problems}\label{sec:volume-constrained-problems}
We refer to the simplified configuration in Figure \ref{fig:domains} (right); here we let $\Gamma=\partial\oomg$, $\omgi=\omgn\cup\omgd$ such that $\omgn\cap\omgd=\emptyset$ and $\Gamma=\Gamma_N\cup\Gamma_D$ such that $\Gamma_N\cap\Gamma_D=\emptyset$. For $s\in V'_c(\oomg)$, $g_n\in V'(\omgn)$ and $v_n\in \widetilde V_c(\oomg)$, we want to solve
\begin{equation}\label{eq:nonlocal-exact}
\left\{\begin{array}{ll}
-\ds\mcL\un       = s   & \xb\in\omg  \\[3mm]
-\ds\mcN(\mcG\un) = \gn & \xb\in\omgn \\[3mm]
\un               = v_n & \xb\in\omgd,
\end{array}\right.
\end{equation} 
where \eqref{eq:nonlocal-exact}$_2$ and \eqref{eq:nonlocal-exact}$_3$ are the nonlocal counterpart of a Neumann and Dirichlet boundary conditions, referred to as Neumann and Dirichlet volume constraints, respectively. More specifically, by composition of the nonlocal interaction and gradient operators we have that \eqref{eq:nonlocal-exact}$_2$ corresponds to
\begin{equation}\label{eq:N}
-\ds\mcN(\mcG\un)(\xb) = \int_\omgomgi (\un(\xb)-\un(\yb))\gamma(\xb,\yb)\,d\yb 
                       = \gn  \quad \forall\,\xb\in\omgn.
\end{equation}
As for local equations, the weak form of \eqref{eq:nonlocal-exact} is obtained by multiplying both sides by a test function $z\in V_c$ and integrating over $\omg$, i.e.
\begin{equation}\label{eq:nonlocal-var}
\ds-\int_\omg \mcL \un z\,d\xb = \int_\omg sz\,d\xb \quad \forall\, z\in V_c(\oomg).
\end{equation}
Using nonlocal integration by parts \cite{Du2013} and the Neumann constraint, \eqref{eq:nonlocal-var} is equivalent to
\begin{equation}\label{eq:nonlocal-var-byparts}
\begin{aligned}
\ds\int_\oomg& \int_\oomg \mcG \un \mcG z \,d\yb\,d\xb = - 
\int_\omgn \mcN(\mcG\un) z\,d\xb + \int_\omg sz\,d\xb  \quad\Rightarrow\\[3mm]
\ds\int_\oomg&\int_\oomg (\un(\xb)-\un(\yb))(z(\xb)-z(\yb))\gamma(\xb,\yb)\,d\yb\,d\xb =
\int_\omgn \gn z\,d\xb + \int_\omg sz\,d\xb  \quad\Rightarrow\\[3mm]
a&(u,z) = F(z),
\end{aligned}
\end{equation}
where the bilinear form and the linear functional are defined as $a(u,z)=\langle u,z\rangle_{V_c}$ and 
$F(v)=\int_\omgn \gn z\,d\xb + \int_\omg sz\,d\xb$. It can be easily shown \cite{Du2012} that for every $\gamma(\cdot,\cdot)$ satisfying the Poincar\'e inequality $a(\cdot,\cdot)$ is coercive and continuous in $V_c(\oomg)\times V_c(\oomg)$ and that $F(\cdot)$ is continuous in $V_c(\oomg)$. Thus, by the Lax-Milgram theorem problem \eqref{eq:nonlocal-var-byparts} is well-posed.

\section{Proposed strategies}\label{sec:strategy}
In engineering applications it is often the case that data are only available on the boundary $\Gamma$ and not in $\omgi$; in particular, most of the times, we are given force or pressure data (i.e. a local Neumann boundary condition) on parts of $\Gamma$. As shown in \cite{Du2012} and as recalled above, this is not enough for the well-posedness of problem \eqref{eq:nonlocal-var-byparts}. 

\smallskip We make the following assumptions.

\smallskip\noindent \textbf{A1} The kernel function $\gamma$ is such that the limit of the nonlocal diffusion operator is the classical Laplacian, i.e. 
\begin{equation}\label{eq:L-limit}
\mcL w(\xb)\to \Delta w(\xb), \quad {\rm as} \quad \veps\to 0. 
\end{equation}
This is obtained by scaling $\gamma$ using some appropriate constant proportional to a power of $\veps$.

\smallskip\noindent \textbf{A2} There exists a local (differential) operator that approximates well enough the nonlocal one when the solution does not feature a nonlocal behavior, i.e. does not exhibit irregularities. Because of assumption \eqref{eq:L-limit} in {\bf A1}, we use the classical Laplacian $\Delta$ as the approximation of $\mcL$ in \eqref{eq:L}. We refer to this model as the {\it surrogate} local model.

\medskip\noindent \textbf{A3} Only the following data are available:

\smallskip\noindent
{\bf 1.} $g_l\in L^2(\Gamma_N)$: {\it local} Neumann boundary data on $\Gamma_N$;\\
{\bf 2.} $v_n\in\widetilde V_c(\oomg)$ on $\omgd$: nonlocal Dirichlet data;\\
{\bf 3.} $s\in V_c'(\oomg)$: forcing term over $\oomg$.

\smallskip\noindent Once again, these do not guarantee existence and uniqueness of a nonlocal solution.

\begin{remark}\label{rem:extension}
We point out that our strategy is readily applicable to a much broader class of nonlocal operators as long as \emph{\bf A2} holds. As an example, this approach could be applied to a linear nonlocal elasticity model (specifically the linear peridynamic solid model \cite{Silling2010}) for which the corresponding surrogate local model is the classical Navier-Cauchy equation of linear elasticity, as the latter is the local limit of the former. 
\end{remark}

Our goal is to design a strategy to automatically convert $g_l$ into a nonlocal volume constraint (either of Neumann or Dirichlet type) on $\omgn$. In the next sections we introduce two conversion approaches and present qualitative comparison results. Note that the conversion problem is an {\it ill-posed} inverse problem as there exists an infinite number of nonlocal data corresponding to $g_l$ for which the associated nonlocal problem is well-posed. However, among all possible choices, we look for a strategy such that the corresponding nonlocal solution, say $\wun$, satisfies
\begin{equation}
\wun \to u_l \;\; {\rm as} \;\; \veps\to 0 \quad {\rm in} \;\; V(\oomg) \;\; {\rm and} \;\; L^2(\oomg),
\end{equation}
where $u_l$ is the solution of the following (surrogate) Poisson equation
\begin{equation}\label{eq:local-Neumann}
\left\{\begin{array}{ll}
-\Delta \ul = s               & \xb\in\oomg \\[3mm]
-\nabla\ul\cdot {\bf n} = g_l & \xb\in\Gamma_N \\[3mm]
\ul = v_n                     & \xb\in\Gamma_D,
\end{array}\right.
\end{equation}
i.e. the solution of the local problem with boundary data as in {\bf A2}. Here, by prescribing the Dirichlet condition on $\Gamma_D$ we are assuming that $v_n|_{\Gamma_D}$ exists and is such that $v_n|_{\Gamma_D}\in H^\frac12(\Gamma_D)$\footnote{Note that, even though this is a regularity requirement (not desirable in nonlocal contexts), we are not assuming $v_n\in H^1(\omgn)$, but only that $v_n$ has a well-defined trace on $\Gamma_D$.}. 

\subsection{Neumann strategy}\label{sec:Neumann-approach}
This is our main and most promising strategy. The key idea is to use the available data in {\bf A3} to solve the surrogate problem in $\oomg$ and utilize the local solution $\ul$ to compute the corresponding force, say $\wgn$, over $\omgn$. It is clear from the right hand side in \eqref{eq:nonlocal-var-byparts} that the nonlocal Neumann data is indeed a forcing term acting on $\omgn$; thus, $\wgn$ will be used as an approximation of $\gn$ to solve \eqref{eq:nonlocal-exact}. We proceed step by step.

\medskip\noindent
{\bf 1N} Solve the surrogate local problem \eqref{eq:local-Neumann}.

\medskip\noindent
{\bf 2N} Compute the forces on $\omgn$ associated with $\ul$. This is achieved by applying the nonlocal Neumann operator $\mcN(\mcG\,\cdot)$ to $\ul$, i.e. $-\mcN(\mcG\ul)(\xb)=\wgn(\xb)$, for $\xb\in\omgn$. This represents an approximation of the nonlocal Neumann data $\gn$. Note that, for the same reasons as for the operator $\mcL$, the Neumann operator $\mcN(\mcG\,\cdot)$ also maps $V$ into $V'$. This implies that 
\begin{equation}\label{eq:wgn-in-L2}
\wgn\in V'(\omgn).
\end{equation}

\medskip\noindent
{\bf 3N} Compute an approximation of the nonlocal solution $\un$, say $\wun$, using $\wgn$ as Neumann data, i.e. solve 
\begin{equation}\label{eq:nonlocal-Neumann}
\left\{\begin{array}{ll}
-\ds\mcL \wun      = s    & \xb\in\omg \\[3mm]
-\ds\mcN(\mcG\wun) = \wgn & \xb\in\omgn \\[3mm]
\wun               = v_n  & \xb\in\omgd.
\end{array}\right.
\end{equation}

\smallskip Because of \eqref{eq:wgn-in-L2}, problem \eqref{eq:nonlocal-Neumann} is well-posed.

\subsection{Dirichlet strategy}\label{sec:Dirichlet-approach}
We present an alternative, and more straightforward, that approach consists in using $u_l$ computed as in {\bf 1N} as Dirichlet volume constraint for the nonlocal problem in $\omgn$. Thus, we have the following procedure.

\medskip\noindent
{\bf 1D} Solve the surrogate local problem \eqref{eq:local-Neumann}.

\medskip\noindent
{\bf 2D} Solve the following nonlocal problem:
\begin{equation}\label{eq:nonlocal-Dirichlet}
\left\{\begin{array}{ll}
-\ds\mcL \widetilde u_{n,D} = s & \xb\in\omg \\[3mm]
\widetilde u_{n,D}          = u_l        & \xb\in\omgn \\[3mm]
\widetilde u_{n,D}          = v_n        & \xb\in\omgd.
\end{array}\right.
\end{equation}
Because of its regularity, $u_l\in\widetilde V(\oomg)$ and, thus, problem \eqref{eq:nonlocal-Dirichlet} is well-posed.

\medskip
This approach cleary delivers a solution that is unable to catch nonlocal behaviors in a neighborhood of the Neumann boundary. This effect is less strong in the previous approach because, instead of prescribing a local constraint on the solution itself, the Neumann approach only prescribes an equivalence of forces allowing the solution to feature a nonlocal behavior. In other words, the locality constraint is weaker. 

This is confirmed by one-dimensional numerical results. We consider $\omg=(0,1)$, $\oomg=(-\veps,1+\veps)$, and $\omgn=(-\veps,0)$. We test both homogeneous and non-homogeneous Neumann conditions; specifically, we consider the following problem settings.

\smallskip\noindent
{\bf A} $s=-12x^2-6/5\veps^2$, $g_l=-4\veps^3$ and $v_n=x^4$;\\
{\bf B} $s=-12x^2-6/5\veps^2$, $g_l=2/5\veps^2(8-13\veps)$ and $v_n = x^4+2x+3/5\veps^2(x^2+2x-3-4\veps-\veps^2)$,

\smallskip\noindent where dependence of the data on $\veps$ is only for testing purposes. We do not specify discretization details as they are not relevant for now. In Figure \ref{fig:comparison} we report $\wun$, $\widetilde u_{n,D}$ and $\ul$ for {\bf A} (left) and {\bf B} (right) in a region around the Neumann boundary. Results show that in both cases the solutions obtained with the Neumann and Dirichlet approaches are significantly different in the zoomed area; in fact, while $\widetilde u_{n,D}$ is, by construction, on top of $u_l$, $\wun$ only reproduces its normal derivative.
\begin{figure}
\centering
\begin{tabular}{cc}
\includegraphics[width=0.4\textwidth]{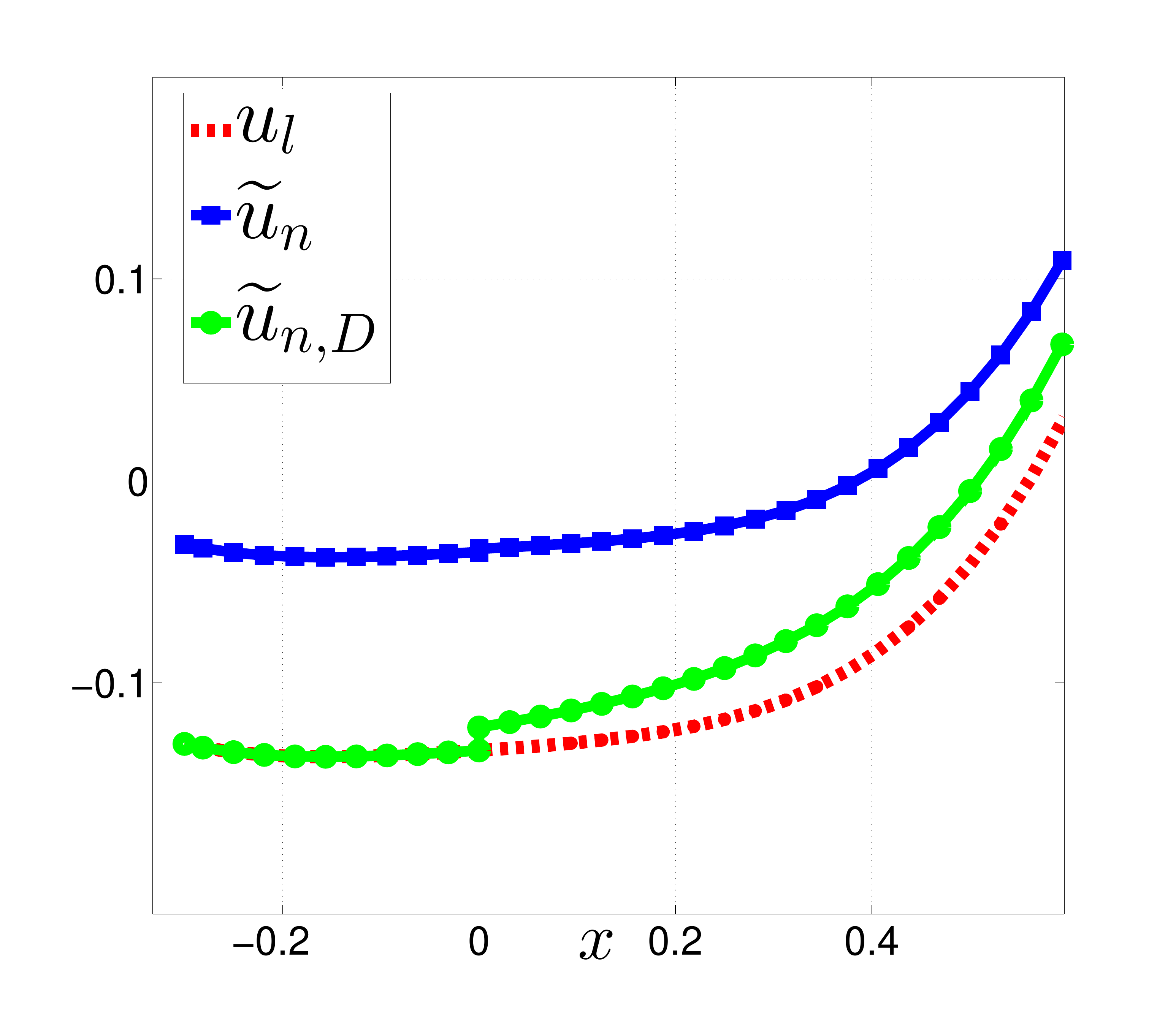}
\includegraphics[width=0.4\textwidth]{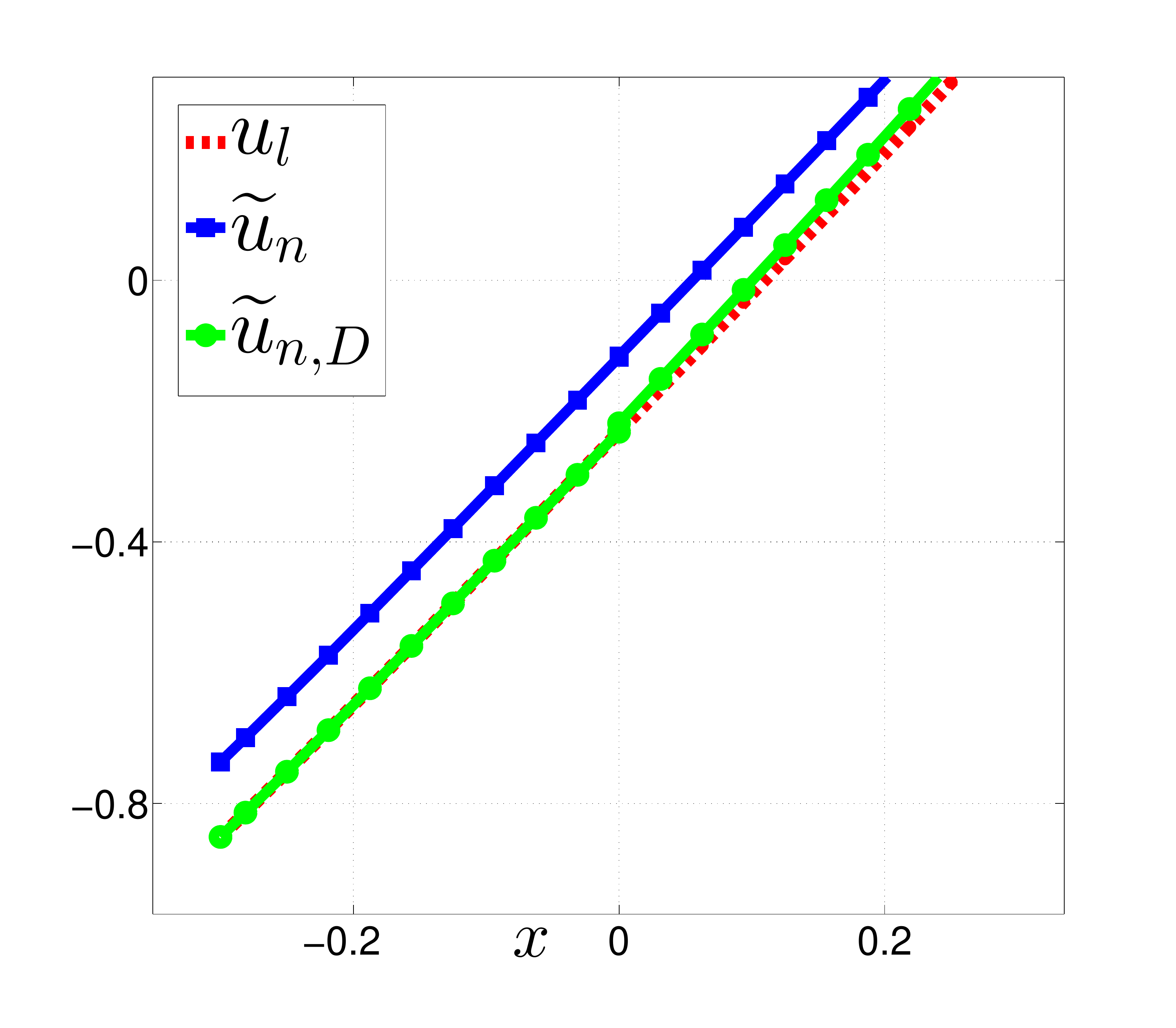}
\end{tabular}
\caption{Comparison of solutions obtained with Neumann ($\wun$) and Dirichlet ($\widetilde u_{n,D}$) strategies for case {\bf A} (left) and {\bf B} (right) around the Neumann boundary.}
\label{fig:comparison}
\end{figure}
Note that when the data are such that local and nonlocal models are equivalent\footnote{For the operators under considerations, we have equivalence for polynomials up to the third order, see numerical experiments in Section \ref{sec:numerical-tests} for an illustration.}, the two approaches coincide and we have that $\widetilde u_{n,D}=\wun=\ul$. This is confirmed by numerical experiments in Section \ref{sec:numerical-tests}.

\section{Convergence to the local limit}\label{sec:local-limit}
In this section we study the limiting behavior of the solution as the nonlocal interactions vanish, i.e. as $\veps\to 0$. We introduce the errors
\begin{equation}\label{eq:local-error}
e_E = |||\wun-\ul|||
\quad {\rm and} \quad
e_0 = \|\un-\wun\|_{0,\oomg}.
\end{equation}

\smallskip The following proposition provides a bound for $e_E$ for the Neumann approach.
\begin{theorem}\label{thm:local-limit}
Let $\veps_0\in(0,\infty)$ and $\mcU_l:=\{\ul\in C^4(\oomg): \ul \hbox{ solves \eqref{eq:nonlocal-Neumann} for }\veps\in(0,\veps_0]\}$ be a family of solutions of \eqref{eq:nonlocal-Neumann}. Then, for all $\ul\in\mcU_l$
\begin{equation}
e_E \leq C \veps^2 \|D^{(4)}\ul\|_{\infty,\oomg}\,,
\end{equation}
where $C$ is a positive constant independent of $\veps$ and $\ul$ and $D^{(4)}$ indicates the 4-th derivative operator.
\end{theorem}
\begin{proof}
Recall that, by definition, $\wun$ and $\ul$ satisfy
\begin{equation}\label{eq:nonlocal-local-Neumann}
\left\{\begin{array}{ll}
-\ds\mcL \wun  = s    = -\Delta \ul   & \xb\in\omg \\[3mm]
-\ds\mcN(\wun) = \wgn = -\ds\mcN(\ul) & \xb\in\omgn \\[3mm]
\wun           = v_n                  & \xb\in\omgd.
\end{array}\right.
\end{equation}
We introduce the following {\it nonlocal} auxiliary problem for the {\it local} solution $\ul$:
\begin{equation}\label{eq:nonlocal-auxiliary}
\left\{\begin{array}{ll}
-\ds\mcL \ul  = \sl  = -\int_\oomg (\ul(\yb)-\ul(\xb))\gamma(\xb,\yb)\,d\yb & \xb\in\omg \\[3mm]
-\ds\mcN(\ul) = \wgn & \xb\in\omgn \\[3mm]
\ul           = v_n  & \xb\in\omgd.
\end{array}\right.
\end{equation}
In order to estimate $e_E$ we first consider the point-wise difference 
$s(\xb)\!-\!\sl(\xb)$. By the Taylor's theorem
\begin{equation}\label{eq:s-sl}
|s(\xb)-\sl(\xb)| =    \left|\int_\oomg (\ul(\yb)-\ul(\xb))\gamma(\xb,\yb)\,d\yb - \Delta\ul \right|
                  \leq \widetilde C \veps^2 |D^{(4)}\ul|_{\infty,\oomg}\,,
\end{equation}
where $\widetilde C$ is a positive constant independent of $\veps$ and $\ul$ and $D^{(4)}$ indicates the 4-th derivative operator.
Next, we consider the weak forms of \eqref{eq:nonlocal-local-Neumann} and \eqref{eq:nonlocal-auxiliary} for the same test function $z\in V_c$; we have
\begin{equation}\label{eq:weak-wun}
\int_\oomg\int_\oomg (\wun(\xb)-\wun(\yb))(z(\xb)-z(\yb))\gamma(\xb,\yb)\,d\yb\,d\xb = 
\int_\omgn \wgn\,z\,d\xb + \int_\omg s\,z\,d\xb,
\end{equation}
\begin{equation}\label{eq:weak-ul}
\int_\oomg\int_\oomg (\ul(\xb)-\ul(\yb))(z(\xb)-z(\yb))\gamma(\xb,\yb)\,d\yb\,d\xb = 
\int_\omgn \gn\,z\,d\xb + \int_\omg \sl\,z\,d\xb.
\end{equation}
Subtraction, yields
\begin{displaymath}
\begin{aligned}
\int_\oomg\int_\oomg(\wun(\xb)-\ul(\xb)-\wun(\yb)+\ul(\yb))(z(\xb)-z(\yb))
\gamma(\xb,\yb)\,d\yb\,d\xb = \int_\omg (s-\sl)\,z\,d\xb.
\end{aligned}
\end{displaymath}
By taking $z=\wun-\ul\in V_c$, we have
\begin{displaymath}
|||\wun-\ul |||^2 \leq \int_\omg (s-\sl)\,(\wun-\ul)\,d\xb
\leq \|s-\sl\|_{0,\omg} \|\wun-\ul\|_{0,\omg} \leq C \veps^2 \|D^{(4)}\ul\|_{\infty,\oomg} |||\wun-\ul |||,
\end{displaymath}
where we omitted the higher order terms because not relevant and where the last inequality follows from the Poincar\'e inequality. Then, the thesis follows by dividing both sides by $|||\wun-\ul |||$.
\end{proof}
\begin{remark}\label{L2rate}
Theorem \ref{thm:local-limit} implies that the convergence rate for the $L^2$ norm of the difference between local and nonlocal solutions, i.e. $e_0$, is at least quadratic. In fact, by the Poincar\'e inequality, we have
\begin{equation}\label{eq:L2rate}
e_0 = \|\wun-\ul\|_{0,\oomg} \leq C_{n,p} |||\wun-\ul||| \leq \widehat C \veps^2 \|D^{(4)}\ul\|_{\infty,\oomg}.
\end{equation}
\end{remark}
\begin{remark}\label{Dirichlet-rate}
A simple modification of the proof of Theorem \ref{thm:local-limit} yields the same convergence result for the Dirichlet strategy\footnote{Simply extend the Dirichlet condition to the whole interaction domain and disregard the term on $\omgn$ in the weak forms.}. The same convergence rate is inherited by the $L^2$ norm, as described in Remark \ref{L2rate}.
\end{remark}
\begin{remark}\label{consistency}
Theorem \ref{thm:local-limit} implies that when the data $g_l$, $s$ and $v_n$ are smooth enough to have $\mcL\ul=\Delta\ul$, then $\wun=\ul$. We use this observation to conduct a consistency test for the proposed conversion method.
\end{remark}

\section{Numerical tests}\label{sec:numerical-tests}
With the purpose of illustrating the theoretical results, in this section we report the results of one-dimensional numerical tests. Even though preliminary, these results are promising and provide the ground work for realistic simulations.

We consider the one-dimensional configuration in Figure \ref{fig:1D-setting}; we let $a=0$, $b=1$, $\Gamma_N=\{x=-\veps\}$, $\Gamma_D=\{x=1+\veps\}$, and
\begin{equation}\label{eq:constant-kernel}
\gamma(x,y)=\dfrac{3}{\veps^3}\,\mcX(|x-y|<\veps).
\end{equation}
This integrable kernel is such that $\mcL w \to \Delta w$ as $\veps\to 0$. In all our tests we discretize the nonlocal equation with the finite element method (FEM) and utilize piecewise linear finite elements. The domain $\oomg$ is partitioned in intervals of the same size $h$. We denote the FEM solutions by $\wun^h$ and $\widetilde u_{n,D}^h$ and introduce the discrete counterparts of the $e_E$ and $e_{0}$, i.e. 
\begin{displaymath}
e_{E,h}=|||w^h-\ul|||
\quad {\rm and} \quad
e_{0,h}=\|w^h-\ul\|,
\end{displaymath}
where $w^h$ is either $\wun^h$ or $\widetilde u_{n,D}^h$. We test both the consistency and the convergence to local limits.
\begin{figure}[t!]
\centering
\includegraphics[width=0.7\textwidth]{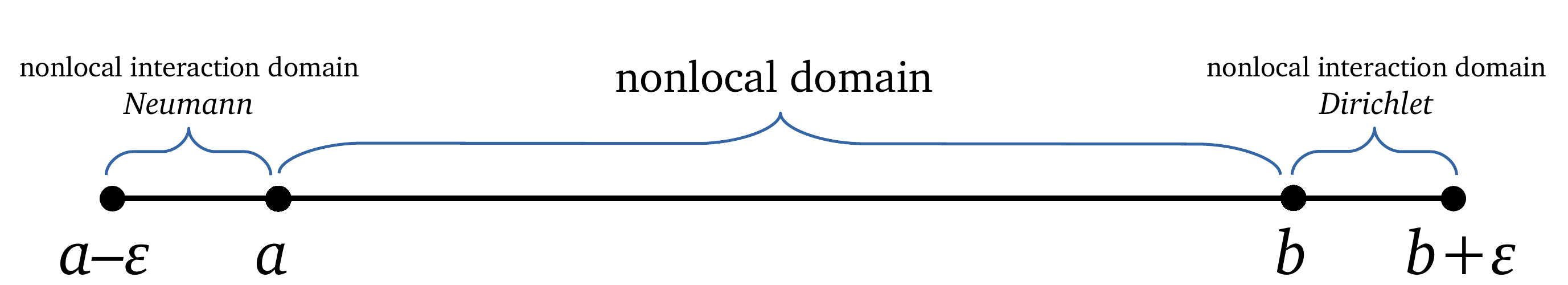}
\caption{One-dimensional configuration.}
\label{fig:1D-setting}
\end{figure}
\begin{remark}
As mentioned in the introduction, our conversion method is not tied to any discretization; in fact, both mesh-free and mesh-based methods can be employed. FEM is, in general, quite expensive for large scale nonlocal simulations, but affordable in a one-simensional setting.
An advantage of using the piecewise linear FEM is the asymptotic compatibility, a property studied in \cite{tian2013analysis, Tian2014} on the robustness of numerical schemes under change of $\veps$.
\end{remark}
\begin{remark}
Note that since we use manufactured solutions for which the local solution can be computed explicitly, we do not approximate the local problem.
\end{remark}

\subsection{Consistency}
We consider local solutions $\ul$ such that $\mcL\ul=\Delta\ul$. According to Remark \ref{consistency} and to the discussion in Section \ref{sec:Dirichlet-approach} the approximate nonlocal solutions $\wun$ and $\widetilde u_{n,D}$ are such that $\wun=\widetilde u_{n,D}=\ul$. Thus, we consider the following problem settings:

\smallskip\noindent
{\bf A} $\ul=x$, $g_l=1$, $v_l=1+\veps$ and $s=0$;\\
{\bf B} $\ul=x^3$, $g_l=3\veps^2$, $v_l=(1+\veps)^3$ and $s=-6x$.

\smallskip\noindent Note that for both {\bf A} and {\bf B} we have that $s=-\mcL u_l=-\Delta u_l$.
For the sake of comparison and to illustrate our theory we consider both the Neumann approach described in Section \ref{sec:Neumann-approach} and the Dirichlet approach described in Section \ref{sec:Dirichlet-approach}. As mentioned above, we expect the two approaches to be equivalent when the local and nonlocal operators are equivalent. Additionally, in case {\bf A} we expect $\wun=\widetilde u_{n,D}=\ul$ and the FEM solution to be $\epsilon$-machine accurate because the exact solution belongs to the space of discretized solutions; in case {\bf B} we also expect $\wun=\widetilde u_{n,D}=\ul$ and $e_{E,h}$ to be independent of $\veps$. Numerical tests confirm that for both Neumann and Dirichlet approaches in case {\bf A} $e_{E,h}=\epsilon$ and in case {\bf B}, $e_{E,h}\cong9$e-5 for a grid of size $h=2^{-6}$ and for several values of $\veps$ for both strategies. In Figure \ref{fig:consistency} we report illustrations of the numerical solutions for both tests cases: $\wun$, $\widetilde u_{nD}$, and $u_l$ are superimposed.
\begin{figure}
\centering
\begin{tabular}{cc}
\includegraphics[width=0.4\textwidth]{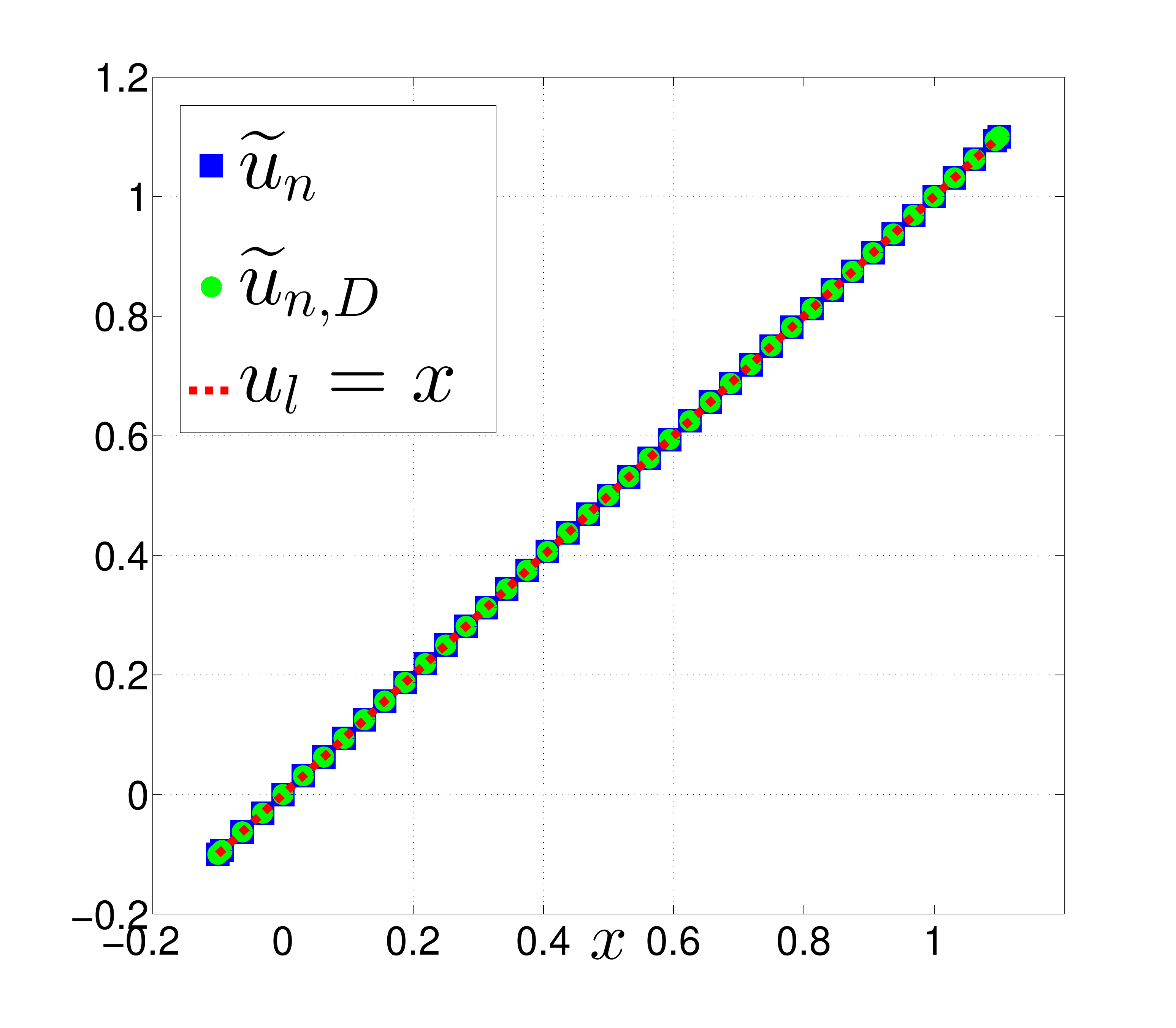}
\includegraphics[width=0.4\textwidth]{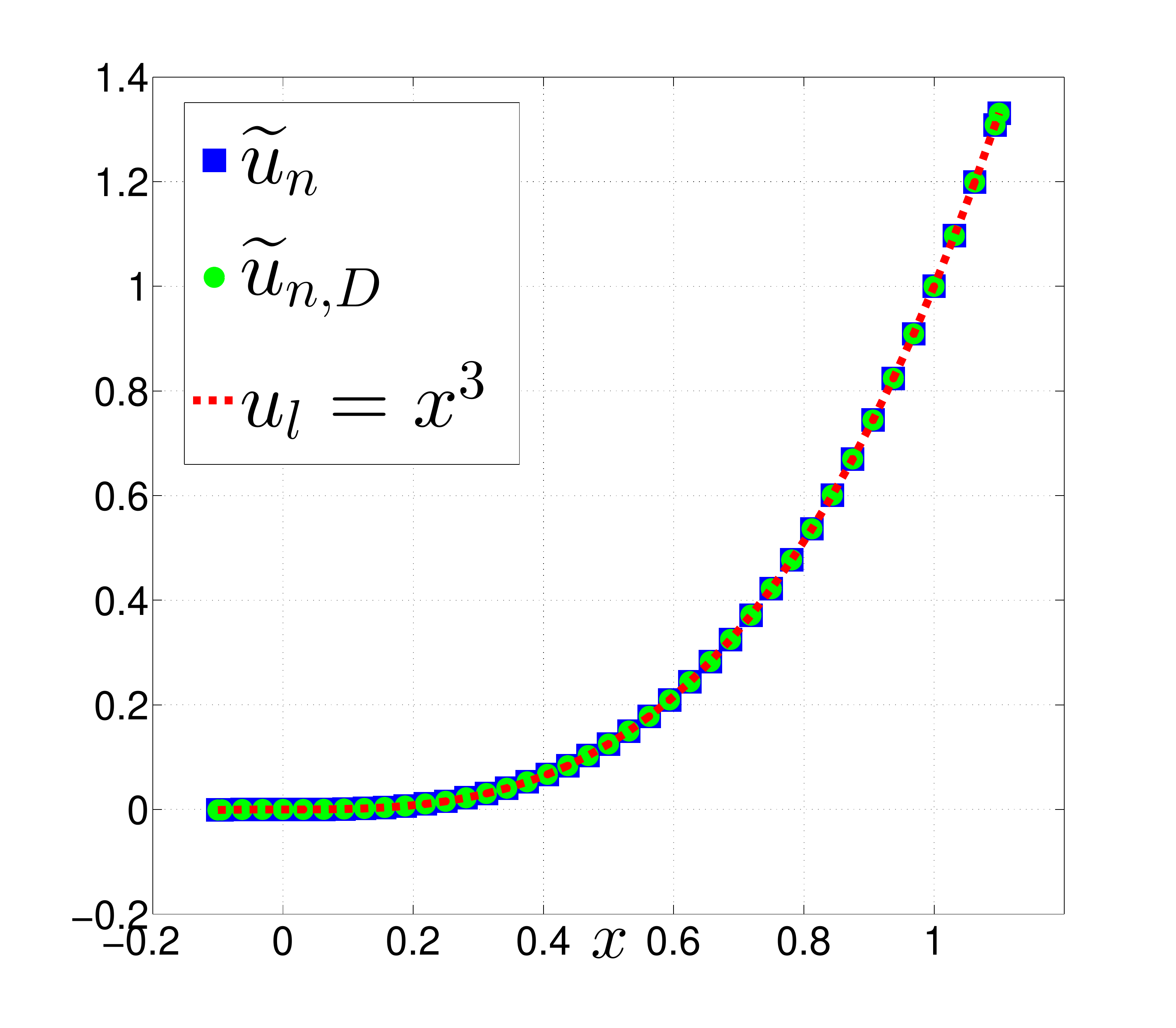}
\end{tabular}
\caption{Nonlocal solutions obtained with the Neumann and Dirichlet strategies and local solution for linear (left) and cubic (right) tests. Up to discretization error, the solutions coincide.}
\label{fig:consistency}
\end{figure}

\subsection{Convergence to local limits}
We perform numerical tests on the convergence of $\wun$ and $\widetilde u_{n,D}$ to the local limit.

We consider the data $g_l=2+5\veps^4$, $v_n=x(2+x^4)$, and $s=-20x^3$; the corresponding local solution is given by $\ul=x(2+x^4)$.

With the purpose of ``hiding'' the discretization error we compute the nonlocal solution on a very fine grid, i.e. $h=2^{-12}$; for decreasing values of $\veps$ we report results in Tables \ref{tab:energy-L2-smallh} and \ref{tab:energy-L2-smallh-D} for the Neumann and Dirichlet strategies respectively. The observed rates for $e_{E,h}$ and $e_{0,h}$ are in alignment with Theorem \ref{thm:local-limit} and Remark \ref{Dirichlet-rate}.

Next, for simoultaneously decreasing values of $\veps$, we test the asymptotic compatibility \cite{tian2013analysis, Tian2014} of our scheme; results are reported in Tables \ref{tab:energy-L2} and \ref{tab:energy-L2-D} for the Neumann and Dirichlet strategies respectively. Also in this case, we have a second order convergence rate. Note that we consider pairs $(h,\veps)=(\veps^2,\veps)$; this choice is motivated by the fact that, for piecewise linear finite element approximations, a linear dependence between $h$ and $\veps$ would compromise the convergence rate of the energy norm due to the influence of the discretization error on the local-limit error. The choice of $h$, makes the discretization error negligible so that the only contribution to the errors is given by the interaction length. As a confirmation, in Table \ref{tab:energy-L2-lin} we report the same results for the pairs $(h,\veps)=(\veps/4,\veps)$, we consider the Neumann approach only. While the convergence of $e_{0,h}$ is still quadratic, the convergence rate of $e_{E,h}$ asymptotically deteriorates (an additional pair with respect to previous tables is added to show deterioration).

Finally, note that, as expected, the errors obtained with the Dirichlet approach are lower that those obtained with the Neumann.
\begin{table}[t!]
\centering
\begin{tabular}{l|ll|ll}
$\veps$  & $e_{E,h}$ & rate  & $e_{0,h}$ & rate \\ \hline
$2^{-2}$ & 9.99e-02  & -     & 7.50e-02  & -    \\ 
$2^{-3}$ & 2.29e-02  & 2.12  & 1.55e-02  & 2.27 \\ 
$2^{-4}$ & 5.48e-03  & 2.06  & 3.50e-03  & 2.15 \\ 
$2^{-5}$ & 1.34e-03  & 2.03  & 8.28e-04  & 2.08 
\end{tabular}
\smallskip
\caption{{\bf Neumann approach:} energy and $L^2$ norm of the difference between local and discretized nonlocal solution for $h=2^{-12}$ and decreasing values of $\veps$.}
\label{tab:energy-L2-smallh}
\end{table}
\begin{table}[t!]
\centering
\begin{tabular}{l|ll|ll}
$\veps$  & $e_{E,h}$ & rate  & $e_{0,h}$ & rate \\ \hline
$2^{-2}$ & 6.95e-02  & -     & 2.48e-02  & -    \\ 
$2^{-3}$ & 1.56e-02  & 2.15  & 5.19e-03  & 2.26 \\ 
$2^{-4}$ & 3.70e-03  & 2.08  & 1.18e-03  & 2.13 \\ 
$2^{-5}$ & 8.99e-04  & 2.04  & 2.82e-04  & 2.07 
\end{tabular}
\smallskip
\caption{{\bf Dirichlet approach:} energy and $L^2$ norm of the difference between local and discretized nonlocal solution for $h=2^{-12}$ and decreasing values of $\veps$.}
\label{tab:energy-L2-smallh-D}
\end{table}      
\begin{table}[t!]
\centering
\begin{tabular}{ll|ll|ll}
$h$       & $\veps$  & $e_{E,h}$ & rate  & $e_{0,h}$ & rate \\ \hline
$2^{-4}$  & $2^{-2}$ & 1.02e-01  & -     & 8.39e-02  & -    \\ 
$2^{-6}$  & $2^{-3}$ & 2.30e-02  & 2.15  & 1.60e-02  & 2.39 \\ 
$2^{-8}$  & $2^{-4}$ & 5.49e-03  & 2.07  & 3.52e-03  & 2.18 \\ 
$2^{-10}$ & $2^{-5}$ & 1.34e-03  & 2.03  & 8.30e-04  & 2.09 
\end{tabular}
\smallskip
\caption{{\bf Neumann approach:} energy and $L^2$ norm of the difference between local and discretized nonlocal solution for simultaneously decreasing values of $\veps$ and $h$ such that $(h,\veps)=(\veps^2,\veps)$.}
\label{tab:energy-L2}
\end{table}
\begin{table}[t!]
\centering
\begin{tabular}{ll|ll|ll}
$h$       & $\veps$  & $e_{E,h}$ & rate  & $e_{0,h}$ & rate \\ \hline
$2^{-4}$  & $2^{-2}$ & 7.32e-02  & -     & 2.96e-02  & -    \\ 
$2^{-6}$  & $2^{-3}$ & 1.58e-02  & 2.22  & 5.42e-03  &  2.45\\ 
$2^{-8}$  & $2^{-4}$ & 3.70e-03  & 2.09  & 1.20e-03  &  2.18\\ 
$2^{-10}$ & $2^{-5}$ & 8.99e-04  & 2.04  & 2.83e-04  &  2.08
\end{tabular}
\smallskip
\caption{{\bf Dirichlet approach:} energy and $L^2$ norm of the difference between local and discretized nonlocal solution for simultaneously decreasing values of $\veps$ and $h$ such that $(h,\veps)=(\veps^2,\veps)$.}
\label{tab:energy-L2-D}
\end{table}
\begin{table}[t!]
\centering
\begin{tabular}{ll|ll|ll}
$h$      & $\veps$  & $e_{E,h}$ & rate & $e_{0,h}$ & rate \\ \hline
$2^{-4}$ & $2^{-2}$ & 1.02e-01  & -    & 8.39e-2   & -    \\ 
$2^{-5}$ & $2^{-3}$ & 2.41e-02  & 2.08 & 1.74e-2   & 2.27 \\ 
$2^{-6}$ & $2^{-4}$ & 6.33e-03  & 1.93 & 3.92e-3   & 2.15 \\ 
$2^{-7}$ & $2^{-5}$ & 1.98e-03  & 1.68 & 9.29e-4   & 2.08 \\
$2^{-8}$ & $2^{-6}$ & 7.75e-04  & 1.35 & 2.26e-4   & 2.05 \\ 
\end{tabular}
\smallskip
\caption{{\bf Neumann approach:} energy and $L^2$ norm of the difference between local and discretized nonlocal solution for simultaneously decreasing values of $\veps$ and $h$ such that $(h,\veps)=(\veps/4,\veps)$.}
\label{tab:energy-L2-lin}
\end{table}

\section{Conclusion}
We introduced a flexible, physically consistent and efficient strategy for the conversion of surface local data into volumetric data in the context of nonlocal modeling and simulation. Our technique does not have regularity constraints on the nonlocal solution, it can be applied in any dimension, and converges to the solution of the corresponding local problem as the nonlocality vanishes. 

More specifically, we achieve second order convergence of the energy norm as the nonlocal interaction vanish in {\it any dimension} and only requiring the local solution to belong to $C^4$ (which can be obtained when the boundary of the domain, the boundary data and the source term are smooth enough).
 
Furthermore, even if numerical results are only in one dimension, the implementation of this approach in two and three dimensions is straightforward and only requires PDE and nonlocal solvers that can be used as black boxes, i.e. the proposed method does not require {\it any} implementation effort. Also, the computational cost is the same as the one required by a single nonlocal simulation. 
 
\section{Acknowledgments}
Marta D'Elia was supported by Sandia National Laboratories (SNL), SNL is a multimission laboratory managed and operated by National Technology and Engineering Solutions of Sandia, LLC., a wholly owned subsidiary of Honeywell International, Inc., for the U.S. Department of Energys National Nuclear Security Administration contract number DE-NA0003525. Specifically, this work was supported through the Sandia National Laboratories Laboratory-directed Research and Development (LDRD) program. This paper describes objective technical results and analysis. Any subjective views or opinions that might be expressed in the paper do not necessarily represent the views of the U.S. Department of Energy or the United States Government. SAND Number: SAND2019-6287. 
The research of Xiaochuan Tian is supported in part by the U.S. NSF grant DMS-1819233. Yue Yu is supported by the U.S. NSF grant DMS-1620434 and the Lehigh faculty research grant.

The authors would like to thank Dr. D. Littlewood (Sandia Natioanl Laboratories, NM) for useful discussions and insights.

\end{document}